\documentclass[11pt]{amsart}
\usepackage{float}
\usepackage[usenames]{color}
\usepackage{graphicx}
\usepackage{amscd,amssymb,verbatim,xcolor}
\usepackage{longtable, multirow}
\usepackage[all]{xy}
\usepackage{amsthm,amssymb,amsmath}
\usepackage{tikz}
\usepackage{mathrsfs,amscd,longtable}
\usepackage{multirow}

\newcommand\fa{{\mathfrak a}}
\newcommand\fb{{\mathfrak b}}
\newcommand\fc{{\mathfrak c}}
\newcommand\fg{{\mathfrak g}}
\newcommand\fh{{\mathfrak h}}
\newcommand\frk{{\mathfrak k}}
\newcommand\fl{{\mathfrak l}}
\newcommand\fm{{\mathfrak m}}
\newcommand\fn{{\mathfrak n}}
\newcommand\fo{{\mathfrak o}}
\newcommand\fp{{\mathfrak p}}
\newcommand\fq{{\mathfrak q}}
\newcommand\fs{{\mathfrak s}}
\newcommand\ft{{\mathfrak t}}
\newcommand\fu{{\mathfrak u}}
\newcommand\fz{{\mathfrak z}}
\newcommand\bB{{\mathbb B}}
\newcommand\bC{{\mathbb C}}
\newcommand\bF{{\mathbb F}}
\newcommand\bH{{\mathbb H}}
\newcommand\bN{{\mathbb N}}
\newcommand\bR{{\mathbb R}}
\newcommand\bW{{\mathbb W}}
\newcommand\bZ{{\mathbb Z}}

\newtheorem{theorem}{Theorem}[section]
\newtheorem{conjecture}[theorem]{Conjecture}

\newtheorem{proposition}[theorem]{Proposition}
\newtheorem{example}[theorem]{Example}
\newtheorem{remark}[theorem]{Remark}

\newtheorem{definition}[theorem]{Definition}
\newcommand{\bb}[1]{\mathbb{#1}}
\newcommand{\mc}[1]{\mathcal{#1}}
\newcommand{\mf}[1]{\mathfrak{#1}}
\begin{document}
\title[Exceptional Richardson Orbits]{Unipotent Representations of Exceptional Richardson Orbits}
\author{Kayue Daniel Wong}
\address{School of Science and Engineering, the Chinese University of Hong Kong, Shenzhen, Guangdong 518172, P.R. China}
\email{kayue.wong@gmail.com}

\maketitle

\begin{abstract}
We study special unipotent representations attached to complex exceptional Richardson orbits. 
As a consequence, we verify a conjecture of Achar and Sommers for these orbits.   
\end{abstract}

\section{Introduction}

In \cite{BV1985}, Barbasch and Vogan studied 
{\bf special unipotent representations} for complex 
simple Lie groups $G$. These representations are 
of interest in various areas of representation 
theory. For instance, they are conjectured to be the
`building blocks' of the unitary dual of $G$. This idea is verified 
by Vogan \cite{V1} for $G = GL(n,\mathbb{F})$ where $\mathbb{F} = \mathbb{C}, \mathbb{R}$
and $\mathbb{H}$, and by Barbasch \cite{B1989} for all complex classical groups.
Moreover, these representations are conjectured 
to be quantization models of special nilpotent orbits (Conjecture \ref{conj:vogan}). 
The conjecture was shown to be true by Barbasch \cite{B2017} for classical nilpotent orbits, 
and by Losev, Mason-Brown and Matvieievskyi \cite{LMM}, \cite{MM} in general 
(Remark \ref{rmk:vogan}). 

\medskip
A main goal of this manuscript is to study a map defined by Lusztig and Vogan
on nilpotent orbits of complex reductive Lie algebras. More explicitly, let $\mathcal{N}$ 
be the set of nilpotent elements of the Lie algebra $\mathfrak{g}$ of $G$,
and $\mathcal{O} \subset \mathcal{N}$ be a nilpotent orbit. For any
$e \in \mathcal{O}$, let $G_e$ be the stabilizer subgroup of $e$ under conjugation, and
$(G_e)^{\vee}$ be the isomorphism class of irreducible, algebraic representations
of $G_e$. They conjectured that there is a bijection between the sets
\begin{equation} \label{eq-lv}
\Phi: \left\{(\mathcal{O}, \sigma)\ |\ \mathcal{O} \subset \mathcal{N}, \sigma \in (G_e)^{\vee}\ \text{for any}\ e \in \mathcal{O} \right\}\ \longrightarrow \Lambda^+(G),
\end{equation}
where $\Lambda^+(G) \subset \mathfrak{t}^*$ be the collection of highest dominant 
weights of finite dimensional representations of $G$. The conjecture is later proved by Bezrukavnikov,
and the image of $\Phi$ is computed explicitly in some cases -- namely,  
Achar \cite{Ac} computed the bijection $G = GL(n,\mathbb{C})$, and the calculations are later 
simplified by Rush \cite{R}. On the other hand, 
Liang-Zhang \cite{LZ} computed the bijection for minimal nilpotent orbits, and
Zhang \cite{Z} computed the case of $G = G_2(\mathbb{C})$.

\medskip
In \cite{W2}, the author defined a map analogous to \eqref{eq-lv}:
$$\Psi: \left\{(\mathcal{O}, \pi)\ |\ \mathcal{O} \subset \mathcal{N}\ \text{special orbit},
\pi \in (\overline{A}(\mathcal{O}))^{\vee} \right\} \longrightarrow \Lambda^+(G),$$
where $\mathcal{O}$ is a special orbit in the sense of Lusztig, and $\overline{A}(\mathcal{O})$ is 
Lusztig's quotient of the component group $A(\mathcal{O}) := G_e/(G_e)^0$ of $G_e$ (c.f. \cite[Chapters 4, 13]{Lu1}). 
To relate $\Psi$ to special unipotent representations and quantization, note that the domain of $\Psi$ 
is in one-to-one correspondence with the special unipotent representations of $G$. 
Indeed, given the validity of the quantization conjecture mentioned above,
then $\Psi$ is equal to $\Phi$ for $\pi \in (G_e)^{\vee}$
that factors through $G_e \to G_e/(G_e)^0 = A(\mathcal{O}) \to \overline{A}(\mathcal{O})$.
In such cases, one can effectively compute (part of) $\Phi$ by understanding $\Psi$. This idea is pursued in \cite{W2}
for classical nilpotent orbits, so that 
one can prove a conjecture by Achar and Sommers \cite{AS} on the image of $\Phi$ for classical Lie groups.

\medskip
In this manuscript, we compute the image of $\Psi$ 
explicitly for all complex exceptional Richardson orbits. 
As a consequence, we prove that the Achar-Sommers conjecture also holds 
for these orbits (Theorem \ref{thm:nonspecial}).

\section{Preliminaries}

\subsection{Special Unipotent Representations}
Let $G$ be a complex simple Lie group, with maximal compact subgroup $K$. We recall the construction 
of irreducible, admissible $(\mathfrak{g}_{\mathbb{C}},K_{\mathbb{C}})$-modules, where $K_{\mathbb{C}} \cong G$ 
is the complexification of $K$.

\smallskip
Let $H=TA$ be the Cartan decomposition of the Cartan subgroup $H$ of $G$, with $\fh=\ft+\fa$. We make the following identifications:
$$
\fh_{\mathbb{C}}\cong \fh \times \fh, \quad \ft_{\mathbb{C}} =\{(x, -x): x\in\fh\} \cong \fh, \quad \fa_{\mathbb{C}} \cong\{(x, x): x\in \fh\}  \cong \fh.
$$
Let $(\lambda_L, \lambda_R)\in \fh^*\times \fh^*$ such that $\eta:=\lambda_L-\lambda_R$ is integral, and write
\begin{equation}\label{def-eta}
\{\eta\} := w'\eta
\end{equation}
as the unique dominant weight to which $\eta$ is conjugate under the action of the Weyl group $w' \in W := W(\fg,\fh)$.
Write $\nu:=\lambda_L + \lambda_R$. We can view $\eta$ as a weight of $T$ and $\nu$ a character of $A$. Put
$$
X(\lambda_L, \lambda_R):={\rm Ind}_B^G(\bC_{\eta}\otimes \bC_{\nu}\otimes {\rm triv})_{K-{\rm finite}},
$$
where $B\supset H$ is the Borel subgroup of $G$ determined by a choice of positive roots $\Delta_G^+$.
Then we define $J(\lambda_L, \lambda_R)$ be the unique irreducible subquotient of $X(\lambda_L, \lambda_R)$ containing the $K_{\mathbb{C}} \cong G$-type $V^G_{\{\eta\}}$. By \cite{Zh}, every irreducible admissible $(\fg, K)$-module has the form $J(\lambda_L, \lambda_R)$. We will refer to the pair $(\lambda_L, \lambda_R)$ as the {\it Zhelobenko parameter} for the module $J(\lambda_L, \lambda_R)$.

\medskip
Among all $J(\lambda_L,\lambda_R)$'s, we focus on the following collection of irreducible modules:
\begin{definition} \label{def:unip}
Let $G$ be a complex Lie group, and ${}^L\mathcal{O} \subset {}^L\mathfrak{g}$ be a special nilpotent orbit in the Langlands dual
of $\mathfrak{g}$. Writing $\mc{O}$ as the Lusztig-Spaltenstein dual of ${}^L\mc{O}$, and $\lambda_{\mc{O}} := \frac{{}^LH}{2}$  
as one half of the semisimple element in a Jacobson-Morozov triple attached to ${}^L\mc{O}$. Then
the {\bf special unipotent representations} attached to ${}^L\mc{O}$ are given by the set
$$\Pi({}^L\mc{O}) := \{ J(\lambda_{\mc{O}},w\lambda_{\mc{O}})\ |\ AV(J(\lambda_{\mc{O}},w\lambda_{\mc{O}})) \subseteq \overline{\mc{O}} \},$$
where $AV(X)$ is the associated variety of any $(\mathfrak{g}_{\mathbb{C}},K_{\mathbb{C}})$-module $X$ (c.f. \cite[Section 2]{V3}).
\end{definition}
By Corollary 5.18 of 
\cite{BV1985}, $J(\lambda_{\mc{O}},w\lambda_{\mc{O}})$ have associated 
variety greater than or equal to $\overline{\mc{O}}$ for all $w \in W$. 
Therefore $\Pi({}^L\mc{O})$ consists of irreducible representations 
$J(\lambda_{\mc{O}},w\lambda_{\mc{O}})$'s whose associated 
varieties are \textit{precisely} equal to $\overline{\mc{O}}$. Note that by Theorem 1.5 
of \cite{BV1985}, $J(\lambda,w\mu) \cong 
J(x\lambda,xw\mu)$ for all $x \in W$. 
So we can begin with any $W$-conjugate of $\lambda$. 

Here is the main theorem of \cite{BV1985}:
\begin{theorem}[\cite{BV1985}, Theorem III] \label{thm-BV}
The cardinality of $\Pi({}^L\mc{O})$ is equal to $|\overline{A}(\mathcal{O})^{\vee}|$, i.e. 
the number of irreducible, non-isomorphic representations of $\overline{A}(\mathcal{O})$.
Moreover, by writing 
\begin{equation} \label{eq-unip}
\Pi({}^L\mc{O}) = \{ X_{\mathcal{O},\pi}\ |\ \pi \in \overline{A}(\mathcal{O})^{\vee}\},
\end{equation}
then the character formula for each $X_{\mathcal{O},\pi}$ can be obtained explicitly.
\end{theorem}
A detailed proof of Theorem \ref{thm-BV} is given by \cite{BV1985} in the case when $\lambda_{\mc{O}}$
is integral, i.e. ${}^L\mc{O}$ is an even orbit. In the Appendix, we complete the proof of the theorem 
for non-integral $\lambda_{\mc{O}}$.

\subsection{Richardson Orbits}
Although one can get the character formula for all $X_{\mathcal{O},\pi}$ by Theorem \ref{thm-BV}, 
its expression can be quite complicated. Therefore, we focus on the case when $\mathcal{O}$ is Richardson, i.e,
there exists a parabolic subalgebra $\mathfrak{p} = \mf{l} + \mf{n}$ such that $\mathcal{O}$ is the dense
orbit in the $G$-saturation $G\cdot \mf{n}$ of $\mf{n}$.
In particular, their closures can be obtained by the image of the moment map
\begin{equation} \label{eq-moment}
\begin{aligned}
\mu: T^*(G/P) \cong G \times_{P} \mathfrak{n} &\to \overline{\mc{O}}
\end{aligned}
\end{equation}
given by $\mu(g,X) := g\cdot X$ for some parabolic group $P = LN$. We would like to investigate when the above map is birational:
\begin{proposition} \label{prop:birational}
Let $\mc{O}$ be an exceptional Richardson orbit other than 
\begin{equation} \label{eq-exception}
A_4+A_1,\ D_5(a_1)\ \text{in}\ E_7, \quad E_6(a_1)+A_1,\ E_7(a_3)\ \text{in}\ E_8.
\end{equation}
Then there exists a parabolic subgroup $P$ such that the moment map $\mu$ in \eqref{eq-moment} is birational onto $\overline{\mc{O}}$.
\end{proposition}
\begin{proof}
Our choices of $P$ for all Richardson orbits other than the four orbits in \eqref{eq-exception}
are listed in Sections 4 -- 8. It is known that if $[G^e:P^e] = 1$, then $\mu$ is birational. 
This holds when $G_e/(G_e)^0 = A(\mc{O}) = 1$. Also, by the results 
in \cite{McG1}, the map is birational when $\mc{O}$ 
is even (so that the Levi subgroup $L$ is given by the 
nodes marked with $0$'s). By checking the tables in 
\cite{CM}, we are only left with 
\begin{equation} \label{eq:birational}
D_4(a_1)+A_1\ \text{in}\ E_7; \ \ D_6(a_1), D_7(a_2)\ \text{in}\ E_8 
\end{equation}
along with the four orbits mentioned in the proposition. 

Indeed, we can apply Lemma 5.5 in \cite{FJLS} to check birationality
of $\mu$, which says that for any $e \in \mc{O}$, 
the number of components in $\mu^{-1}(e)$ is equal to 
$$\sum_{\psi \in \widehat{A(e)}} \dim V_{\psi} 
\cdot [\sigma_{\mc{O},\psi}: \mathrm{Ind}_{W(L)}^{W}(\mathrm{sgn})],$$ 
where $\sigma_{\mc{O},\psi}$ is the Springer 
representation $H^{top}(\mathcal{B}_e)^{\psi}$. 
One can use the table in \cite{Alv} to check that the above value
is equal to $1$ for the orbits in \eqref{eq:birational}. 
Therefore $\mu$ is birational for all the orbits stated in the proposition.
%
\end{proof}

For all exceptional Richardson $\mathcal{O}$ other than the four orbits in \eqref{eq-exception},
Theorem 9.11 of \cite{BV1985} provides a nice character formula for all $\Pi({}^L\mathcal{O})$. 
\begin{theorem} \label{thm:BV9}
Let $\mc{O}$ be an exceptional Richardson orbit not equal to four orbits in \eqref{eq-exception}, and $\mf{l}$ be 
the Levi subalgebra of the parabolic subgroup corresponding to $\mc{O}$ given by Proposition \ref{prop:birational}, then the following statements hold:
\begin{itemize}
\item[(a)] The number of elements in 
$$P(\mc{O}) := \{ w\lambda_{\mathcal{O}}\ |\  w\lambda_{\mathcal{O}}|_{\mf{l}}\ \text{is dominant and regular}\}$$ 
is equal to the number of conjugacy classes of $\overline{A}(\mc{O})$. In particular, there exists a unique $\lambda_1$ such that $(\lambda_1,\alpha^{\vee}) = 1$ for all simple roots in $\mf{l}$, and $\lambda_1 - \lambda$ is a sum of positive roots for all $\lambda \in P(\mc{O})$.
\item[(b)] The elements in $P(\mc{O}) = \{ \lambda_1, \lambda_2, \dots, \lambda_r\}$ can be arranged such that $|\lambda_i - \lambda_1| < |\lambda_{i+1} - \lambda_1|$ for all $i$.
\item[(c)] The representations $J(\lambda_1, \lambda_i)$ exhaust all special unipotent representations in $\Pi({}^L\mc{O})$. More precisely, suppose $\overline{A}(\mc{O}) \cong S_k$ for $2 \leq k \leq 5$, and $C_1 = \langle k \rangle > \dots > C_r = \langle 1^k \rangle$ is the total ordering of partitions of $k$ (such ordering exists since $k \leq 5$). Then $\pi_i \in \overline{A}(\mc{O})^{\vee}$
are parametrized by the partition $C_i$, with
$$X_{\mc{O},\pi_{i}} := J(\lambda_1, \lambda_i) \cong \mathrm{Ind}_L^G(V_i)$$
for some finite-dimensional representations $V_i$ of $L$ with highest weight $(\lambda_1 - \lambda_i)|_{\mf{l}}$.
In particular, when $i = 1$, then $\pi_1$ is the trivial representation of $\overline{A}(\mc{O})^{\wedge}$ and $V_1$ is the trivial representation of $L$.
\end{itemize}
\end{theorem}
\begin{proof}
The values of $\lambda_i$ in (a) and (b) are computed in Sections 4 -- 8. Assuming the results of (a) and (b) hold, then for each $\lambda_i$, let $I_i$ be as defined in Section 9 of \cite{BV1985} such that the character formula of 
$I_i$ is given by
\begin{equation} \label{eq-charform}
\sum_{w \in W(L)} \mathrm{sgn}(w) X(\lambda_1,w\lambda_i),
\end{equation}
so that $I_i \cong \mathrm{Ind}_L^G(V^L_{\lambda_1 - \lambda_i})$ as $K_{\bb{C}} \cong G$-modules. Obviously, we have $AV(I_i) = \overline{\mathrm{Ind}_{\mf{l}}^{\mf{g}}(0)} = \overline{\mc{O}}$, and its composition factors must consist of special unipotent representations attached to $\mc{O}$. Using the arguments in 9.11 - 9.21 of \cite{BV1985}, one can see $I_i = J(\lambda_1,\lambda_i)$ and (c) follows.
\end{proof}

\subsection{Vogan's Conjecture on Quantization}
As mentioned in the introduction, the special unipotent representations $\Pi({}^L\mathcal{O})$
are conjectured to be the `quantizations' of (local systems of) $\mathcal{O}$. More precisely, we
have the following conjecture by Vogan in the 1980s:
\begin{conjecture}[\cite{V3}, Conjecture 12.1] \label{conj:vogan}
Let $\mc{O}$ be a complex special nilpotent orbit, and
$\pi$ be an irreducible representation of $\overline{A}(\mc{O})$.
Then there exists an irreducible representation $W_{\psi}$
of $G_e$ that factors through $G_e/(G_e)_0 = A(\mc{O})$ such that:
$$X_{\mc{O},\pi}|_{K_\mathbb{C}} \cong R(\mc{O},\psi) \cong \mathrm{Ind}_{G_e}^G(W_{\psi}),$$
where $R(\mc{O},\psi)$ is the global section of 
the vector bundle $G \times_{G_e} W_{\psi} \to G/G_e 
\cong \mc{O}$. 
\end{conjecture}

One can use the results in \cite{S3} to see that
the irreducible $L$-modules $V_{i}$ in Theorem \ref{thm:BV9}(c) are all `lifts' 
of some $W_{\psi_i} \in A(\mc{O})^{\wedge}$ in Conjecture 
\ref{conj:vogan}. Namely, the restricted representation $V_{i}|_{L_e}$ factors through 
$L_e/(L_e)^0 \cong G_e/(G_e)^0$, and is isomorphic to $W_{\psi_i}$.
This gives an evidence on the validity of Conjecture \ref{conj:vogan}.
Indeed, for $\mathrm{triv} \in \overline{A}(\mc{O})^{\vee}$, one can easily verify the conjecture for 
Richardson orbits:
\begin{theorem} \label{thm:vogan}
Let $\mc{O}$ be an exceptional Richardson orbit. Then $X_{\mc{O},\mathrm{triv}}
\cong R(\mc{O})$, i.e. Conjecture \ref{conj:vogan}
holds for all exceptional Richardson orbits.
\end{theorem}
\begin{proof}
When $\mc{O}$ is not equal to the four orbits in \eqref{eq-exception}, the map $\mu: T^*(G/P) \to \overline{\mc{O}}$ is the normalization of the orbit closure $\overline{\mc{O}}$. By standard arguments in algebraic geometry (see \cite{JN2004} for example), $R(\mc{O}) \cong \bb{C}[T^*(G/P)]$. By a result of \cite{McG1}, the latter is isomorphic to $\mathrm{Ind}_L^G(\mathrm{triv})$ as $G$-modules. This means $R(\mc{O}) \cong \mathrm{Ind}_L^G(\mathrm{triv}) \cong X_{\mc{O},\mathrm{triv}}$ and therefore Theorem \ref{thm:vogan} holds for these orbits. We will get the same conclusion for the four orbits in \eqref{eq-exception} in Section 9.
\end{proof}

\begin{remark} \label{rmk:vogan}
As mentioned in the introduction, Conjecture \ref{conj:vogan} is proved in full generality 
in \cite{LMM} and \cite{MM}. More explicitly, a definition of \emph{unipotent representations}
for $G$-equivariant covers of nilpotent orbits is given in \cite[Definition 1.4.1]{LMM}.
It is shown in \cite{LMM} (for classical groups) and \cite{MM} (for exceptional and spin groups) 
that all special unipotent representations $X_{\mathcal{O},\pi}$ in Definition \ref{def:unip}
are unipotent representations attached to some covers
of $\mathcal{O}$. Then by \cite[Section 6.7]{LMM}, Conjecture \ref{conj:vogan} holds for all
special unipotent representations.
\end{remark}

%
%

\subsection{The Lusztig-Vogan map}
We now give a precise description of the map \eqref{eq-lv} given by Theorem 8.2 of \cite{V1998}:
\begin{definition} \label{def-lv}
For each nilpotent element $\mathcal{O} \subset \mathfrak{g}$, and $W_{\sigma} \in (G_e)^{\vee}$ for any $e \in \mathcal{O}$, write
\begin{align} \label{eq:lusztigvogan}
R(\mathcal{O},\sigma) \cong \mathrm{Ind}_{G^e}^G(W_{\sigma}) \cong \sum_{\lambda \in \Lambda^+(G)} m_{\mathcal{O},\sigma}(\lambda) \mathrm{Ind}_{T}^{G}(e^{\lambda}),
\end{align}
where all but finitely many $m_{\mathcal{O},\sigma}(\lambda) \in \mathbb{Z}$ are zero. Then the {\bf Lusztig-Vogan map}
$$\Phi: \left\{(\mathcal{O}, \sigma)\ |\ \mathcal{O} \subset \mathcal{N}, \sigma \in (G_e)^{\vee}\ \text{for any}\ e \in \mathcal{O} \right\} \rightarrow \Lambda^+(G)$$
is defined by $\Phi(\mathcal{O},\sigma) = \lambda_{max},$
where $\lambda_{\max}$ is the largest dominant element in 
Equation \eqref{eq:lusztigvogan} such that $m_{\mathcal{O},\sigma}(\lambda_{\max}) \neq 0$.
\end{definition}
%

The paper \cite{W2} studies $\Phi$
for all classical nilpotent orbits $\mathcal{O}$, and all $(G_e)^{\vee}$ that factors through $(G_e/(G_e)^0)^{\vee} = A(\mc{O})^{\vee}$,
and consequently proved the
following conjecture of Achar and Sommers for classical groups:
\begin{conjecture}[\cite{AS}] \label{conj-AS}
Let $\mc{O}$ be a special nilpotent orbit. Consider the \textbf{special piece} $\mc{S}_{{}^L\mc{O}}$ (\cite{Lu2}) of the Lusztig-Spaltenstein dual ${}^L\mc{O}$ of $\mc{O}$:
$$S_{{}^L\mc{O}} := \{{}^L\mc{P} \subseteq \overline{{}^L\mc{O}} | {}^L\mc{P} \nsubseteq \overline{{}^L\mc{O}_{sp}}\ \text{for other special } {}^L\mc{O}_{sp} \subsetneq \overline{{}^L\mc{O}} \}.$$
For each ${}^L\mc{P} \in \mc{S}_{{}^L\mc{O}}$, let $H_{{}^L\mc{P}}$ be the semi-simple element of a Jacobson-Morozov triple
corresponding to ${}^L\mc{P}$. Then there exists $W_{\psi} \in A(\mc{O})^{\wedge}$ (which can be seen as a representation of
$G_e$ that factors through $A(\mc{O})$) such that 
$$\Phi(\mc{O},\psi) = H_{{}^L\mc{P}}.$$
\end{conjecture}
Let us describe how the above conjecture is proved in \cite{W2} for classical groups. In \cite[Definition 1.5]{W2}, the author defined the map
\begin{equation} \label{eq-psi}
\Psi: \{(\mc{O},\pi)\ |\ \mc{O}\ \text{special nilpotent orbit,}\ \pi \in \overline{A}(\mc{O})^{\wedge}\} \longrightarrow \Lambda^+(G),
\end{equation}
as follows: Suppose the character formula $X_{\mc{O},\pi}$ is given by
$$\sum_{w \in W} a_{\mc{O},\pi}(w)X(\lambda_{\mc{O}},w\lambda_{\mc{O}}).$$
On the level of Grothendieck groups, the restriction of $X_{\mc{O},\pi}$ into $K_{\mathbb{C}}$-modules is of the form
\begin{equation} \label{eq:unip2}
X_{\mc{O},\pi}|_{K_{\mathbb{C}}} = \sum_{w \in W} a_{\mc{O},\pi}(w)\mathrm{Ind}_T^G(\{\lambda_{\mc{O}} - w\lambda_{\mc{O}}\})
= \sum_{\lambda \in \Lambda^+(G)} b_{\mathcal{O},\pi}(\lambda)\mathrm{Ind}_T^G(e^{\lambda}),
\end{equation}
where $\{\lambda_{\mc{O}} - w\lambda_{\mc{O}}\}$ is as defined in \eqref{def-eta}, and $b_{\mathcal{O},\pi}(\lambda) \in \mathbb{Z}$. Then we define $\Psi(\mc{O},\pi) := \lambda_{\max}$, 
where $\lambda_{\max}$ is the largest dominant element in 
\eqref{eq:unip2} such that $b_{\mc{O},\pi}(\lambda_{\max})$ $\neq$ $0$.

\medskip
Since Conjecture \ref{conj:vogan} holds for all nilpotent orbits, 
one can relate the two maps $\Psi$ and $\Phi$ by Equations \eqref{eq:lusztigvogan}
and \eqref{eq:unip2}. As a consequence, one can verify Conjecture \ref{conj-AS} 
by computing $\Psi$. This is done for all classical special nilpotent orbits $\mathcal{O}$ 
and all $\pi \in \overline{A}(\mc{O})^{\vee}$ in \cite{W2}. In this manuscript, we apply the same strategy for exceptional Richardson orbits and obtain the following:
\begin{theorem} \label{thm:nonspecial}
Let $G$ be a complex exceptional Lie group of adjoint type, 
and $\mc{O}$ be a Richardson orbit. For every semisimple 
element $H_{{}^L\mathcal{P}}$ in a Jacobson-Morozov triple 
corresponding to ${}^L\mc{P} \in \mc{S}_{{}^L\mc{O}}$, there exists $\pi \in \overline{A}(\mc{O})^{\wedge}$
such that 
$$\Psi(\mc{O},\pi) = H_{{}^L\mc{P}}.$$
Consequently, Conjecture \ref{conj-AS} holds for all exceptional Richardson orbits.
\end{theorem}


%


\section{Proof of Theorem \ref{thm:nonspecial} - General Case}
In Sections 4 -- 8, we write down the details for Theorem \ref{thm:BV9} and Theorem \ref{thm:nonspecial} for all  Richardson orbits other than the four orbits in
\eqref{eq-exception}. Namely, for each exceptional Richardson orbit $\mc{O}$, we list
\begin{itemize}
\item the Levi subalgebra $\mf{l}$ where $\mc{O}$ is induced from, by specifying a sub-diagram of the Dynkin diagram of $\mf{g}$;
\item all irreducible representations $\pi_i$ of $\overline{A}(\mc{O}) \cong S_k$ ($k = 2,3,4,5$) in terms of partitions of $k$;
\item the values of $\lambda_i$ appearing in Theorem \ref{thm:BV9}(a)--(b).
\end{itemize}
This verifies Theorem \ref{thm:BV9}. 

\medskip
We now turn our attention to Theorem \ref{thm:nonspecial}. By Equation \eqref{eq-charform} and \eqref{eq:unip2}, the image of $\Psi$ can be easily computed as
$$\Psi(\mc{O},\pi_{i}) = \{\lambda_1 - w_L(\lambda_{i})\}$$ 
for all $\pi_i \in \overline{A}(\mc{O})^{\vee}$, where $w_L$ is the longest element in the Weyl group $W(L)$. The values of $\Psi(\mc{O},\pi_{i})$ are recorded on the second last column of the tables in terms of the weighted Dynkin diagram of ${}^L\mf{g}$. And the last column records the orbit ${}^L\mc{P}$ whose Dynkin element $H_{{}^L\mc{P}}$ is given by the previous column, which verifies Theorem \ref{thm:nonspecial} for all but four exceptional Richardson orbits given in \eqref{eq-exception}.\\

\section{$G_2$ Orbits}
The exceptional group $G_2$ has three Richardson orbits.  Fix the Dynkin diagram of $\mf{g}$ by: 
.

{\scriptsize
\begin{center}
\begin{longtable}{|l|l|l|l|l|l|}
 \hline
$\mc{O}$ & $\mf{l}$ & $\pi_i$ & $\lambda_i$ &  $\Psi(\mc{O},\pi_i)$ &  ${}^L\mathcal{P}$ \\\hline \endhead \hline
\endfoot

$G_2$ & $\phi$ & $\langle 1 \rangle$ & $(0,0,0)$ &  \begin{tikzpicture} \draw[fill=black] (0,0) circle [radius=.1] node [above] {0} -- (1,0) circle [radius=.1] node [above] {0} node [midway,below] {$3$} ; \end{tikzpicture}
&  $0$ \\\hline

\multirow{3}{*}{$G_2(a_1)$}& \multirow{3}{*}{\begin{tikzpicture} \draw[fill=black] (0,0)       circle [radius=.1] node [above] {(1,-1,0)}   ; \end{tikzpicture}}
& $\langle3\rangle$
& $(0,-1,1)$
& \begin{tikzpicture} \draw[fill=black] (0,0) circle [radius=.1] node [above] {0} -- (1,0) circle [radius=.1] node [above] {1} node [midway,below] {$3$} ; \end{tikzpicture}
&  $A_{1}$ \\\cline{3-6}
																											
& & $\langle2,1\rangle$
& $(1,-1,0)$
& \begin{tikzpicture} \draw[fill=black] (0,0) circle [radius=.1] node [above] {1} -- (1,0) circle [radius=.1] node [above] {0} node [midway,below] {$3$} ; \end{tikzpicture}
& $\widetilde{A}_{1}$ \\\cline{3-6}

& & $\langle1,1,1\rangle$
& $(1,0,-1)$
& \begin{tikzpicture} \draw[fill=black] (0,0) circle [radius=.1] node [above] {0} -- (1,0) circle [radius=.1] node [above] {2} node [midway,below] {$3$} ; \end{tikzpicture}
& $G_2(a_1)$ \\\hline

$0$ & \begin{tikzpicture} \draw[fill=black] (0,0) circle [radius=.1] node [above] {(-2,1,1)} -- (1.5,0) circle [radius=.13] node [above] {(1,-1,0)} node [midway,below] {$3$} ; \end{tikzpicture}
 & $\langle 1 \rangle$ & $(-1,-2,3)$ &  \begin{tikzpicture} \draw[fill=black] (0,0) circle [radius=.1] node [above] {2} -- (1,0) circle [radius=.1] node [above] {2} node [midway,below] {$3$} ; \end{tikzpicture}
&  $G_2$ \\\hline

\end{longtable}
\end{center}}

\vspace{-5mm}

\section{$F_4$ Orbits}
The exceptional group $F_4$ has nine Richardson orbits. Fix the Dynkin diagram of $\mf{g}$ by %
 with simple roots:
$$
.
$$
{\scriptsize
\begin{center}
\begin{longtable}{|l|l|l|l|l|l|}
 \hline
$\mc{O}$ & $\mf{l}$ & $\pi_i$ & $\lambda_i$ &  $\Psi(\mc{O},\pi_i)$ &  ${}^L\mathcal{P}$ \\\hline \endhead \hline
\endfoot

$F_4$
& $\phi$
& $\langle 1 \rangle$
& $(0,0,0,0)$
&  \begin{tikzpicture} \draw[fill=black] (0,0)       circle [radius=.05] node [above] {{\small 0}} -- (0.5,0)       circle [radius=.05] node [above] {{\small 0}} --       node [midway,below] {$2$} (1,0)       circle [radius=.05] node [above] {{\small 0}} -- (1.5,0)       circle [radius=.05] node [above] {{\small 0}}
      ; \end{tikzpicture}
&  $0$ \\\hline

\multirow{2}{*}{$F_4(a_1)$}& \multirow{2}{*}{\begin{tikzpicture} \draw[fill=black] (0.9,0)       circle [radius=.05] node [above] {{\small 3}} ; \end{tikzpicture}}
& $\langle2\rangle$
& $(\frac{1}{2},0,0,\frac{1}{2})$
& \begin{tikzpicture} \draw[fill=black] (0,0)       circle [radius=.05] node [above] {{\small 0}} -- (0.5,0)       circle [radius=.05] node [above] {{\small 0}} --       node [midway,below] {$2$} (1,0)       circle [radius=.05] node [above] {{\small 0}} -- (1.5,0)       circle [radius=.05] node [above] {{\small 1}}
      ; \end{tikzpicture}
&  $A_{1}$  \\\cline{3-6}

& & $\langle1,1\rangle$
& $(\frac{-1}{2},0,0,\frac{1}{2})$
& \begin{tikzpicture} \draw[fill=black] (0,0)       circle [radius=.05] node [above] {{\small 1}} -- (0.5,0)       circle [radius=.05] node [above] {{\small 0}} --       node [midway,below] {$2$} (1,0)       circle [radius=.05] node [above] {{\small 0}} -- (1.5,0)       circle [radius=.05] node [above] {{\small 0}}
      ; \end{tikzpicture}
&  $\widetilde{A_{1}}$ \\
\end{longtable}
\end{center}
}

\vspace{-10mm}

\section{$E_6$ Orbits}
The adjoint exceptional group $E_6$ has fifteen Richardson orbits. Fix the Dynkin diagram %
 with simple roots:
\begin{eqnarray*}
.
\end{eqnarray*}

{\scriptsize
\begin{center}
\begin{longtable}{|l|l|l|l|l|l|}
 \hline
$\mc{O}$ & $\mf{l}$ & $\pi_i$ & $\lambda_i$ &  $\Psi(\mc{O},\pi_i)$ &  ${}^L\mathcal{P}$ \\\hline \endhead \hline
\endfoot

$E_6$
& $\phi$
& $\langle 1 \rangle$
& $(0,0,0,0,0,0)$
&  \begin{tikzpicture} \draw[fill=black] (0,0)       circle [radius=.05] node [below] {0} -- (0.5,0)       circle [radius=.05] node [below] {0} -- (1,0)       circle [radius=.05] node [below] {0} -- (1,0) --++ (90:0.4)       circle [radius=.05] node [above] {0} (1,0) --++ (0:0.5)       circle [radius=.05] node [below] {0} -- (2,0)       circle [radius=.05] node [below] {0}       ; \end{tikzpicture}
&  $0$ \\\hline

$E_6(a_1)$
& \begin{tikzpicture} \draw[fill=black] (0.5,0)       circle [radius=.05] node [below] {5}   ; \end{tikzpicture}
& $\langle 1 \rangle$
& $(0,\frac{1}{2},\frac{-1}{2},0,0,0)$
&  \begin{tikzpicture} \draw[fill=black] (0,0)       circle [radius=.05] node [below] {0} -- (0.5,0)       circle [radius=.05] node [below] {0} -- (1,0)       circle [radius=.05] node [below] {0} -- (1,0) --++ (90:0.4)       circle [radius=.05] node [above] {1} (1,0) --++ (0:0.5)       circle [radius=.05] node [below] {0} -- (2,0)       circle [radius=.05] node [below] {0}       ; \end{tikzpicture}
&  $A_1$ \\\hline

$D_5$
& \begin{tikzpicture} \draw[fill=black]  (0.5,0)       circle [radius=.05] node [below] {5}        (1,0.5)       circle [radius=.05] node [above] {2}   ; \end{tikzpicture}
& $\langle 1 \rangle$
& $(0,\frac{1}{2},\frac{-1}{2},\frac{1}{2},\frac{1}{2},0)$
&  \begin{tikzpicture} \draw[fill=black] (0,0)       circle [radius=.05] node [below] {1} -- (0.5,0)       circle [radius=.05] node [below] {0} -- (1,0)       circle [radius=.05] node [below] {0} -- (1,0) --++ (90:0.4)       circle [radius=.05] node [above] {0} (1,0) --++ (0:0.5)       circle [radius=.05] node [below] {0} -- (2,0)       circle [radius=.05] node [below] {1}       ; \end{tikzpicture}
&  $2A_1$ \\
\hline

\end{longtable}
\end{center}
}


\section{$E_7$ Orbits}
The adjoint exceptional group $E_7$ has twenty seven Richardson orbits excluding $A_4+A_1$ and $D_5(a_1)$. Fix the Dynkin diagram %
 with simple roots:
\begin{eqnarray*}
.
\end{eqnarray*}

{\scriptsize
\begin{center}
\begin{longtable}{|l|l|l|l|l|l|}
 \hline
$\mc{O}$ & $\mf{l}$ & $\pi_i$ & $\lambda_i$ &  $\Psi(\mc{O},\pi_i)$ &  ${}^L\mathcal{P}$ \\\hline \endhead \hline
\endfoot

$E_7$
& $\phi$
& $\langle 1 \rangle$
& $(0,0,0,0,0,0,0,0)$
&  \begin{tikzpicture} \draw[fill=black]
(0,0)       circle [radius=.05] node [below] {0} --
(0.25,0)       circle [radius=.05] node [below] {0} --
(0.5,0)       circle [radius=.05] node [below] {0} --
(0.75,0)       circle [radius=.05] node [below] {0} --
(0.75,0) --++ (90:0.23)       circle [radius=.05] node [above] {0}
(0.75,0) --++ (0:0.23)       circle [radius=.05] node [below] {0} --
(1.25,0)       circle [radius=.05] node [below] {0}
; \end{tikzpicture}
&  $0$ \\\hline

$E_7(a_1)$
& \begin{tikzpicture} \draw[fill=black]
(0.75,0)       circle [radius=.05] node [below] {4}
; \end{tikzpicture}
& $\langle 1 \rangle$
& $(0,0,0,\frac{1}{2},\frac{-1}{2},0,0,0)$
&  \begin{tikzpicture} \draw[fill=black]
(0,0)       circle [radius=.05] node [below] {0} --
(0.25,0)       circle [radius=.05] node [below] {0} --
(0.5,0)       circle [radius=.05] node [below] {0} --
(0.75,0)       circle [radius=.05] node [below] {0} --
(0.75,0) --++ (90:0.23)       circle [radius=.05] node [above] {0}
(0.75,0) --++ (0:0.23)       circle [radius=.05] node [below] {0} --
(1.25,0)       circle [radius=.05] node [below] {1}
; \end{tikzpicture}
&  $A_1$ \\
\hline

\end{longtable}
\end{center}}


\section{$E_8$ Orbits}
The exceptional group $E_8$ has thirty two Richardson orbits excluding $E_6(a_1)+A_1$ and $E_7(a_3)$. Fix the Dynkin diagram %
 with simple roots:
{\scriptsize
\begin{eqnarray*}
.
\end{eqnarray*}}

{\scriptsize
\begin{center}
\begin{longtable}{|l|l|l|l|l|l|}
 \hline
$\mc{O}$ & $\mf{l}$ & $\pi_i$ & $\lambda_i$ &  $\Psi(\mc{O},\pi_i)$ &  ${}^L\mathcal{P}$ \\\hline \endhead \hline
\endfoot

$E_8$
& $\phi$
& $\langle 1 \rangle$
& $(0,0,0,0,0,0,0,0)$
&  \begin{tikzpicture} \draw[fill=black]
(-0.25,0)       circle [radius=.05] node [below] {0} --
(0,0)       circle [radius=.05] node [below] {0} --
(0.25,0)       circle [radius=.05] node [below] {0} --
(0.5,0)       circle [radius=.05] node [below] {0} --
(0.75,0)       circle [radius=.05] node [below] {0} --
(0.75,0) --++ (90:0.3)       circle [radius=.05] node [above] {0}
(0.75,0) --++ (0:0.25)       circle [radius=.05] node [below] {0} --
(1.25,0)       circle [radius=.05] node [below] {0}
; \end{tikzpicture}
&  $0$ \\\hline

$E_8(a_1)$
& \begin{tikzpicture} \draw[fill=black]
(-0.25,0)   circle [radius=.05] node [below] {8}
; \end{tikzpicture}
& $\langle 1 \rangle$
& $(\frac{1}{2},\frac{-1}{2},0,0,0,0,0,0)$
&  \begin{tikzpicture} \draw[fill=black]
(-0.25,0)       circle [radius=.05] node [below] {1} --
(0,0)       circle [radius=.05] node [below] {0} --
(0.25,0)       circle [radius=.05] node [below] {0} --
(0.5,0)       circle [radius=.05] node [below] {0} --
(0.75,0)       circle [radius=.05] node [below] {0} --
(0.75,0) --++ (90:0.3)       circle [radius=.05] node [above] {0}
(0.75,0) --++ (0:0.25)       circle [radius=.05] node [below] {0} --
(1.25,0)       circle [radius=.05] node [below] {0}
; \end{tikzpicture}
&  $A_1$ \\
\hline

\end{longtable}
\end{center}}

\section{Proof of Theorem \ref{thm:nonspecial} - the four exceptional cases}
We now finish the proof of Theorem \ref{thm:nonspecial} for 
$\mc{O}$ equal to the four orbits in \eqref{eq-exception}. Note that all these 
orbits have Lusztig's quotient $\overline{A}(\mc{O})$ $\cong$ $S_2$, and the special piece
$\mc{S}_{{}^L\mc{O}} = \{{}^L\mc{O}\}$ only contains one element. Therefore, one only 
needs to find one $\pi \in \overline{A}(\mc{O})^{\vee}$ such that $\Psi(\mc{O},\pi) = H_{{}^L\mc{O}}$
in order to verify Theorem \ref{thm:nonspecial}.

Note that the orbits $A_4+A_1$ in $E_7$ and $E_6(a_1)+A_1$ 
in $E_8$ are called {\it exceptional} in Section 4 of \cite{BV1985}.\\

\subsection{$\mc{O} = A_4+A_1$ in $E_7$}
By checking the tables of \cite{AL} directly
(the calculations for other orbits can be found in the Appendix), the two left cell representations attached to $\mc{O}$ are equal to $J_{D_6 + A_1}^{E_7}(\sigma_1)$, $J_{D_6+A_1}^{E_7}(\sigma_2)$, where $\sigma_1$, $\sigma_2$ are the two left cell representations attached to $\mc{O}'$ $=$ $[332211] + [11]$ in $D_6 + A_1$. Using Proposition 6.6 of \cite{BV1985}, the character formulas for $X_{\mc{O},\pi}$ can be derived from that of $\mc{O}'$ in
$$ D_6 + A_1 = \begin{tikzpicture} \draw[fill=black] (0,0)       circle [radius=.13] node [above] {\scriptsize{$(1,-1,0,0,0,0,0,0)$}} -- (1.4,0)       circle [radius=.13] node [below] {\scriptsize{$(0,1,-1,0,0,0,0,0)$}} -- (2.8,0)       circle [radius=.13] node [above] {\scriptsize{$(0,0,1,-1,0,0,0,0)$}} --  (4.2,0)       circle [radius=.13] node [below] {\scriptsize{$(0,0,0,1,-1,0,0,0)$}} --
(4.2,0)--++ (90:0.8)    circle [radius=.13] node [above] {\scriptsize{$(\frac{-1}{2},\frac{-1}{2},\frac{-1}{2},\frac{-1}{2},\frac{1}{2},\frac{1}{2},\frac{1}{2},\frac{1}{2})$}}
(4.2,0)--++ (0:1.4)    circle [radius=.13] node [above] {\scriptsize{$(0,0,0,0,1,-1,0,0)$}}   (7.5,0)       circle [radius=.13] node [below] {\scriptsize{$(0,0,0,0,0,0,1,-1)$}}       ; \end{tikzpicture}.$$
Using Theorem 3.4 in \cite{W2}, $\Psi(\mc{O}', \langle 2 \rangle) = \begin{tikzpicture} \draw[fill=black] (0,0)       circle [radius=.05] node [below] {0} -- (0.3,0)       circle [radius=.05] node [below] {2} -- (0.6,0)       circle [radius=.05] node [below] {0} -- (0.9,0)       circle [radius=.05] node [below] {2} -- (0.9,0) --++ (90:0.25)       circle [radius=.05] node [above] {0} (0.9,0) --++ (0:0.3)       circle [radius=.05] node [below] {0}  (1.75,0)       circle [radius=.05] node [below] {2}       ; \end{tikzpicture}\ ;\ \Psi(\mc{O}', \langle 1^2 \rangle) = \begin{tikzpicture} \draw[fill=black] (0,0)       circle [radius=.05] node [below] {0} -- (0.3,0)       circle [radius=.05] node [below] {2} -- (0.6,0)       circle [radius=.05] node [below] {1} -- (0.9,0)       circle [radius=.05] node [below] {0} -- (0.9,0) --++ (90:0.25)       circle [radius=.05] node [above] {1} (0.9,0) --++ (0:0.3)       circle [radius=.05] node [below] {1}  (1.75,0)       circle [radius=.05] node [below] {2}       ; \end{tikzpicture}$. Therefore, one can compute that
$$\Psi(A_4+A_1,\langle 2 \rangle) = \{(1,1,-1,-1,-3,-3,4,2)\} = \left(\frac{1}{2}, \frac{1}{2}, \frac{-1}{2},\frac{-1}{2},\frac{-3}{2},\frac{-3}{2},\frac{-5}{2},\frac{11}{2}\right),$$
$$\Psi(A_4+A_1,\langle 1^2 \rangle) = \{(1,1,-1,-2,-2,-3,4,2)\} = \left(\frac{1}{2}, \frac{-1}{2}, \frac{-1}{2},\frac{-1}{2},\frac{-1}{2},\frac{-3}{2},\frac{-5}{2},\frac{11}{2}\right).$$
In terms of Dynkin diagram, we have
$$\Psi(A_4+A_1, \langle 2 \rangle) = \begin{tikzpicture} \draw[fill=black] (0,0)       circle [radius=.05] node [below] {0} -- (0.3,0)       circle [radius=.05] node [below] {1} -- (0.6,0)       circle [radius=.05] node [below] {0} -- (0.9,0)       circle [radius=.05] node [below] {1} -- (0.9,0) --++ (90:0.25)       circle [radius=.05] node [above] {0} (0.9,0) --++ (0:0.3)       circle [radius=.05] node [below] {0} -- (1.5,0)       circle [radius=.05] node [below] {1}       ; \end{tikzpicture}\ ;\ \Psi(A_4+A_1, \langle 1^2 \rangle) = \begin{tikzpicture} \draw[fill=black] (0,0)       circle [radius=.05] node [below] {1} -- (0.3,0)       circle [radius=.05] node [below] {0} -- (0.6,0)       circle [radius=.05] node [below] {0} -- (0.9,0)       circle [radius=.05] node [below] {0} -- (0.9,0) --++ (90:0.25)       circle [radius=.05] node [above] {1} (0.9,0) --++ (0:0.3)       circle [radius=.05] node [below] {1} -- (1.5,0)       circle [radius=.05] node [below] {1}       ; \end{tikzpicture}.$$
Note that we have $\Psi(A_4+A_1, \langle 2 \rangle) = H_{A_4+A_1} = H_{{}^L(A_4+A_1)}$ and Theorem \ref{thm:nonspecial} holds.

\subsection{$\mc{O} = D_5(a_1)$ in $E_7$} \label{subsec:9.2}
As in the above subsection, one can derive the character formulas of $X_{\mc{O},\pi}$ from that of $\mc{O}'$ $=$ $[332211]$ in $D_6$. More precisely, we have
$\Psi(\mc{O}', \langle 2 \rangle) = \begin{tikzpicture} \draw[fill=black] (0,0)       circle [radius=.05] node [below] {0} -- (0.3,0)       circle [radius=.05] node [below] {2} -- (0.6,0)       circle [radius=.05] node [below] {0} -- (0.9,0)       circle [radius=.05] node [below] {2} -- (0.9,0) --++ (90:0.25)       circle [radius=.05] node [above] {0} (0.9,0) --++ (0:0.3)       circle [radius=.05] node [below] {0}  (1.75,0)       circle [radius=.05] node [below] {0}       ; \end{tikzpicture}\ ;\ \Psi(\mc{O}', \langle 1^2 \rangle) = \begin{tikzpicture} \draw[fill=black] (0,0)       circle [radius=.05] node [below] {0} -- (0.3,0)       circle [radius=.05] node [below] {2} -- (0.6,0)       circle [radius=.05] node [below] {1} -- (0.9,0)       circle [radius=.05] node [below] {0} -- (0.9,0) --++ (90:0.25)       circle [radius=.05] node [above] {1} (0.9,0) --++ (0:0.3)       circle [radius=.05] node [below] {1}  (1.75,0)       circle [radius=.05] node [below] {0}       ; \end{tikzpicture}$. Therefore, one can compute that
$$\Psi(D_5(a_1),\langle 2 \rangle) = \{(1,1,-1,-1,-3,-3,3,3)\} = (1,1,-1,-1,-1,-1,-3,5),$$
$$\Psi(D_5(a_1),\langle 1^2 \rangle) = \{(1,1,-1,-2,-2,-3,3,3)\} = (1,0,0,-1,-1,-1,-3,5).$$
In terms of Dynkin diagram, we have
$$\Psi(D_5(a_1), \langle 2 \rangle) = \begin{tikzpicture} \draw[fill=black] (0,0)       circle [radius=.05] node [below] {0} -- (0.3,0)       circle [radius=.05] node [below] {2} -- (0.6,0)       circle [radius=.05] node [below] {0} -- (0.9,0)       circle [radius=.05] node [below] {0} -- (0.9,0) --++ (90:0.25)       circle [radius=.05] node [above] {0} (0.9,0) --++ (0:0.3)       circle [radius=.05] node [below] {0} -- (1.5,0)       circle [radius=.05] node [below] {2}       ; \end{tikzpicture}\ ;\ \Psi(D_5(a_1), \langle 1^2 \rangle) = \begin{tikzpicture} \draw[fill=black] (0,0)       circle [radius=.05] node [below] {1} -- (0.3,0)       circle [radius=.05] node [below] {0} -- (0.6,0)       circle [radius=.05] node [below] {1} -- (0.9,0)       circle [radius=.05] node [below] {0} -- (0.9,0) --++ (90:0.25)       circle [radius=.05] node [above] {0} (0.9,0) --++ (0:0.3)       circle [radius=.05] node [below] {0} -- (1.5,0)       circle [radius=.05] node [below] {2}       ; \end{tikzpicture}.$$
Note that $\Psi(D_5(a_1), \langle 2 \rangle) = H_{A_4} = H_{{}^L(D_5(a_1))}$, and Theorem \ref{thm:nonspecial} holds.

\subsection{$\mc{O} = E_6(a_1)+A_1$ in $E_8$}
As before, the character formulas for $X_{\mc{O},\pi}$ can be derived from that of $\mc{O}'$ $=$ $D_5(a_1)$ $+$ $[11]$ in $E_7+A_1$, which, by last subsection, can be derived from that of $[332211]$ $+$ $[11]$ in $D_6 + A_1$. Using the same argument as above, we have
$$\Psi(E_6(a_1)+A_1, \langle 2 \rangle) = \begin{tikzpicture} \draw[fill=black] (-0.3,0)       circle [radius=.05] node [below] {1} -- (0,0)       circle [radius=.05] node [below] {0} -- (0.3,0)       circle [radius=.05] node [below] {1} -- (0.6,0)       circle [radius=.05] node [below] {0} -- (0.9,0)       circle [radius=.05] node [below] {0} -- (0.9,0) --++ (90:0.25)       circle [radius=.05] node [above] {0} (0.9,0) --++ (0:0.3)       circle [radius=.05] node [below] {0} -- (1.5,0)       circle [radius=.05] node [below] {1}       ; \end{tikzpicture}\ ;\ \Psi(E_6(a_1)+A_1, \langle 1^2 \rangle) = \begin{tikzpicture}  \draw[fill=black] (-0.3,0)       circle [radius=.05] node [below] {1} --  (0,0)       circle [radius=.05] node [below] {1} -- (0.3,0)       circle [radius=.05] node [below] {0} -- (0.6,0)       circle [radius=.05] node [below] {0} -- (0.9,0)       circle [radius=.05] node [below] {0} -- (0.9,0) --++ (90:0.25)       circle [radius=.05] node [above] {1} (0.9,0) --++ (0:0.3)       circle [radius=.05] node [below] {0} -- (1.5,0)       circle [radius=.05] node [below] {0}       ; \end{tikzpicture}.$$

We have $\Psi(E_6(a_1)+A_1, \langle 2 \rangle) = H_{A_4+A_1} = H_{{}^L(E_6(a_1)+A_1)}$ and Theorem \ref{thm:nonspecial} follows.

\subsection{$\mc{O} = E_7(a_3)$ in $E_8$} \label{subsec:9.4}
Finally, we can derive the character formula of $X_{\mc{O},\pi}$ from that of $\mc{O}'$ $=$ $[332211]$ in $D_6$. As in the previous subsections, we have
$$\Psi(E_7(a_3), \langle 2 \rangle) = \begin{tikzpicture} \draw[fill=black] (-0.3,0)       circle [radius=.05] node [below] {2} -- (0,0)       circle [radius=.05] node [below] {0} -- (0.3,0)       circle [radius=.05] node [below] {0} -- (0.6,0)       circle [radius=.05] node [below] {0} -- (0.9,0)       circle [radius=.05] node [below] {0} -- (0.9,0) --++ (90:0.25)       circle [radius=.05] node [above] {0} (0.9,0) --++ (0:0.3)       circle [radius=.05] node [below] {0} -- (1.5,0)       circle [radius=.05] node [below] {2}       ; \end{tikzpicture}\ ;\ \Psi(E_7(a_3), \langle 1^2 \rangle) = \begin{tikzpicture}  \draw[fill=black] (-0.3,0)       circle [radius=.05] node [below] {2} --  (0,0)       circle [radius=.05] node [below] {0} -- (0.3,0)       circle [radius=.05] node [below] {0} -- (0.6,0)       circle [radius=.05] node [below] {0} -- (0.9,0)       circle [radius=.05] node [below] {0} -- (0.9,0) --++ (90:0.25)       circle [radius=.05] node [above] {0} (0.9,0) --++ (0:0.3)       circle [radius=.05] node [below] {1} -- (1.5,0)       circle [radius=.05] node [below] {0}       ; \end{tikzpicture}.$$
We have $\Psi(E_7(a_3), \langle 2 \rangle) = H_{A_4} = H_{{}^L(E_7(A_3))}$, and Theorem \ref{thm:nonspecial} holds.\\

This finishes the proof of Theorem \ref{thm:nonspecial}. \qed

\begin{remark} \label{rmk:voganconj4}
For the four orbits above, the sum of the two special 
unipotent representations attached to $A_4+A_1$, $D_5(a_1)$ 
in $G = E_7$, and $E_6(a_1)+A_1$, $E_7(a_3)$ in $G = E_8$ 
are isomorphic to $\mathrm{Ind}_{L}^{G}(\mathrm{triv})$, where the 
semisimple part of $L$ is of type $A_4+A_1$ for the orbits 
$A_4+A_1$, $E_6(a_1)+A_1$, and of type $A_4$ for $D_5(a_1)$, 
$E_7(a_3)$. In fact, the orbit $[332211]$ in $D_6$ that shows 
up in all the above examples is also a Richardson orbit, induced 
from $GL(5) \times SO(2)$. The sum of the two special 
unipotent representations attached to $[332211]$ is 
isomorphic to $\mathrm{Ind}_{GL(5) \times SO(2)}^{SO(12)}(\mathrm{triv})$.\\

There are two consequences of this observation:
\begin{itemize}
\item[(a)] Non-birationality of $\mu$: We give an 
alternative proof of Proposition \ref{prop:birational} 
for these four orbits -- For example, suppose on the 
contrary that $\mu: T^*(G/P) \to \overline{D_5(a_1)}$ 
is birational with the semisimple part of $L$ is of 
Type $A_4$, then $R(D_5(a_1))$ $\cong$ $\mathbb{C}[T^*(G/P)]$ $\cong 
\mathrm{Ind}_{L}^{E_7}(\mathrm{triv})$ $\cong$ $X_{D_5(a_1),\mathrm{triv}} 
\oplus X_{D_5(a_1),\mathrm{sgn}}$.

However, the multiplicity of $\overline{D_5(a_1)}$ 
in $R(D_5(a_1))$ is one, while the multiplicity 
in $X_{D_5(a_1),\mathrm{triv}} 
\oplus X_{D_5(a_1),\mathrm{sgn}}$ is greater than one.
\item[(b)] An alternative proof of Conjecture \ref{conj:vogan} for these orbits: By \cite{V3}, 
we have $X_{\mc{O},\mathrm{triv}} = R(\mc{O}) - 
Y_1$ and $X_{\mc{O},\mathrm{sgn}} = R(\mc{O},\mathrm{sgn}) 
- Y_2$ for some genuine $K_{\bb{C}}$-modules $Y_1$, $Y_2$. 
Therefore,
\[
\mathrm{Ind}_{L}^{G}(\mathrm{triv}) = X_{\mc{O},\mathrm{triv}} \oplus X_{\mc{O},\mathrm{sgn}} = R(\mc{O}) + R(\mc{O},\mathrm{sgn}) - Y_1 - Y_2.
\]
By Proposition 4.6.1 of \cite{B2017}, the left hand side is equal to $R(\mc{O}) + R(\mc{O},\mathrm{sgn})$. So $Y_1 = Y_2 = 0$ and hence we have $X_{\mc{O},\mathrm{triv}} \cong R(\mc{O})$ and $X_{\mc{O},\mathrm{sgn}} \cong R(\mc{O},\mathrm{sgn})$.
\end{itemize}
\end{remark}
The method we used in (b) in the above remark 
can also be applied to verify Conjecture \ref{conj:vogan} 
for other Richardson orbits with non-birational 
moment maps. For example, let $\mc{O} = F_4(a_3)$ 
in $F_4$. If we take the the parabolic subgroup 
whose Levi is of type $A_2 + \widetilde{A_1}$, 
then one can show that $X_{\mc{O},\langle31\rangle} 
\cong R(\mc{O},\langle31\rangle)$. On the other hand, 
if we take the the parabolic subgroup whose Levi 
is of type $B_2$, then one can show that $X_{\mc{O},
\langle22\rangle} \cong R(\mc{O},\langle22\rangle)$.

\bigskip
\noindent {\bf Funding}:\ The author is supported by 
Shenzhen Science and Technology Innovation Committee grant
(no. 20220818094918001).

\appendix

\section{Special Unipotent Representations for Non-even ${}^L\mc{O}$}
Based on the results of \cite{BV1985} on the structure of $\Pi({}^L\mc{O})$
for even ${}^L\mc{O}$, we give a sketch proof of Theorem \ref{thm-BV} when the special orbit ${}^L\mc{O}$ is not even
(or equivalently $\lambda_{\mc{O}}$ is not integral). 
Although it is known by the experts that the results in \cite{BV1985} carry over to all $\lambda_{\mc{O}}$ 
regardless of its integrality, we do not find a proof of such result in the literature.

\medskip
\noindent {\it Sketch Proof of Theorem \ref{thm-BV} for non-even ${}^L\mc{O}$:}
\begin{itemize}
\item[(I)] Let $\lambda := \lambda_{\mc{O}}$. Consider the Lie subalgebra $\mf{g}'$ 
of $\mf{g}$ with roots $\alpha$ satisfying $\langle 
\alpha^{\vee}, \lambda \rangle \in \bb{Z}$. 
It turns out that there is a choice of simple roots of $\mf{g}'$
such that their inner products with $\lambda$ 
is equal to either $0$ or $1$.

\item[(II)] Upon restricting to $\mf{g}'$, the result in \cite{BV1985} applies, 
and we can obtain {\it integral} special unipotent representations with 
infinitesimal character $\lambda = \frac{{}^LH'}{2}$
for some {\it even} orbit ${}^L\mc{O}' \in {}^L\mf{g}'$.

\item[(III)] More explicitly, let $\mc{O}' \in \mf{g}'$ be the Lusztig-Spaltenstein
dual of ${}^L\mc{O}'$, and $V^L(\mc{O}')$ be the left cell representation
of $W'$ denoted by $V^L(w_0w_{\mc{O}'})$ in \cite{BV1985}. By
Chapter 4 of \cite{Lu1}, the irreducible representations in 
$V^L(\mc{O}')$ are parametrized by $\{\sigma_x'\ |\ x \in [\overline{A}(\mc{O}')]\}$ 
(i.e. partitions $[n_1^{m_1} n_2^{m_2} \dots]$ of $k$ if $\overline{A}(\mc{O}') \cong S_k$).

By Theorem III of \cite{BV1985}, the unipotent 
representations $\Pi({}^L\mc{O}')$ are parametrized by $\pi \in \overline{A}(\mc{O}')^{\vee}$, and their character formulas
are given by
$$J_{G'}(\lambda,w_{\pi}\lambda) 
= \frac{1}{|\overline{A}(\mc{O}')|}\sum_{x \in [\overline{A}(\mc{O}')]} \mathrm{tr}_{\pi}(x) |x| \sum_{w' \in W'} \mathrm{tr}_{\sigma_x'}(w')X_{G'}(\lambda,w'\lambda).$$
Under the above identification of elements in $V^L(\mc{O}')$, the
representation $\sigma_{sp}' := \sigma_{[1^k]}'$ is special, and is the Springer
representation of $\mc{O}'$. All other irreducible representations $J_{G'}(\lambda,w\lambda)$
with our choice of $\lambda$ have character formulas involving special representations $\tau'_{sp} \in (W')^{\vee}$, 
where $\mathrm{deg}(\tau'_{sp}) < \mathrm{deg}(\sigma_{sp}')$.

\item[(IV)] By a version of the Kazhdan-Lusztig conjecture (c.f. \cite[Theorem 4.2]{B1989}), 
\begin{equation} \label{eq:charxg}
J(\lambda,w_{\pi}\lambda) 
= \frac{1}{|\overline{A}(\mc{O}')|}\sum_{x \in [\overline{A}(\mc{O}')]} \mathrm{tr}_{\pi}(x) |x| \sum_{w' \in W'} \mathrm{tr}_{\sigma_x'}(w')X(\lambda,w'\lambda)
\end{equation}
are character formulas of irreducible representations of $G$ for all $\pi \in \overline{A}(\mc{O}')^{\vee}$. 
The same result goes for character formulas of other irreducible representations $J_{G'}(\lambda,w\lambda)$
that involves $\tau'_{sp} \in (W')^{\vee}$. 

\item[(V)] Let $J_{W'}^{W}(\sigma_{sp}')$ is the $J$-truncated induction defined in \cite[Section 4.3]{Lu1}.
It contains a unique special representation $\sigma_{sp}$
with $\mathrm{deg}(\sigma_{sp}) = \mathrm{deg}(\sigma_{sp}')$ (and similarly we have 
$\tau_{sp} \in J_{W'}^{W}(\tau'_{sp})$ for all special representations $\tau'_{sp} \in (W')^{\vee}$). By studying the $W \times W$-module 
$\mathbb{C}[W]$  (c.f. \cite[Proposition 6.6]{BV1985}), these character 
formulas in $G$ contain expressions involving $\sigma_{sp}$ (or
$\tau_{sp}$), and their associated varieties can be determined
by the Springer correspondence of $\sigma_{sp}$ (or $\tau_{sp}$).
Since $\mathrm{deg}(\tau_{sp}) < \mathrm{deg}(\sigma_{sp})$,
the irreducible representations $J(\lambda,w_{\pi}\lambda)$
in \eqref{eq:charxg} are the ones with smallest possible associated variety
for our choice of $\lambda$.

\item[(VI)] We verify the following statements in the next couple of sections:
\begin{itemize}
\item[$\bullet$] For all orbits we are studying, 
$\overline{A}(\mc{O}) \cong \overline{A}(\mc{O}')$,
and $\sigma_{sp} = J_{W'}^W(\sigma_{sp}')$ is non-zero and irreducible. 
More precisely, $\sigma_{sp}$ is the Springer representation 
attached to $\mc{O}$, the Lusztig-Spaltenstein
dual of ${}^L\mc{O}$.
\item[$\bullet$] If $\mc{O}$ is not equal to 
the three exceptional orbits ($A_4 + A_1$ in $E_7$, $A_4+A_1$ and $E_6(a_1) + A_1$
in $E_8$) listed in \cite[Definition 4.5]{BV1985}, $J_{W'}^{W}(\sigma_x')$ is non-zero and irreducible 
for all $\sigma_x' \in V^L(\mc{O}')$. Since $V^L(\mc{O}) = 
J_{W'}^W(V^L(\mc{O}'))$ by \cite[Proposition 3.18]{BV1983}, 
the elements of $V^L(\mc{O})$ can be parametrized by
$$x \in [A(\mc{O})] \cong [A(\mc{O}')] \longleftrightarrow  \sigma_x := J_{W'}^{W}(\sigma_x') \in V^L(\mc{O}).$$
This identification of elements in $V^L(\mc{O})$
matches with that of \cite{Lu1}.
\end{itemize}
The first point guarantees that all $J(\lambda,w_{\pi}\lambda)$
have associated variety equal to $\overline{\mc{O}}$ 
(see also Theorem 5.2 of \cite{McG2}), and they have
the same cardinality as the number of 
irreducible representations of $\overline{A}(\mc{O})$, i.e. all $J(\lambda,w_{\pi}\lambda) \in \Pi({}^L\mc{O})$
and $|\Pi({}^L\mc{O})| = |\overline{A}(\mc{O})^{\vee}|$.
The second point and \cite[Proposition 6.6]{BV1985} 
reformulate Equation \eqref{eq:charxg} into
$$J(\lambda,w_{\pi}\lambda) 
= \frac{1}{|\overline{A}(\mc{O})|}\sum_{x \in [\overline{A}(\mc{O})]} \mathrm{tr}_{\pi}(x) |x| \sum_{w \in W} \mathrm{tr}_{\sigma_x}(w)X(\lambda,w\lambda).$$
In other words, the character formula of $J(\lambda,w_{\pi}\lambda)$
is given precisely by Theorem III of \cite{BV1985}, and the result follows. 
\end{itemize}

\subsection{Classical Lie algebras}
We apply the strategy in the previous section to classify all classical special unipotent representations. Recall from \cite{CM}
that each classical nilpotent orbit ${}^L\mc{O} \subset {}^L\mf{g}$ can be uniquely characterized by
a partition, except in Type $D$ that there are two orbits corresponding to a single
\emph{very even} partition of the form
$[2a \geq 2a \geq 2b \geq 2b \geq \cdots \geq 2c \geq 2c]$ (the \emph{very even orbits}). Since all very even orbits are even, 
the results in \cite{BV1985} can be applied directly. 

For convenience, we denote all non-very even ${}^L\mc{O}$ by their corresponding partition.
\begin{proposition}[\cite{CM}, Proposition 6.3.7] \label{prop:class} \mbox{}\\
\begin{itemize}
\item {\bf ${}^L\mf{g}$ of Type $B$:} Let ${}^L\mc{O} = [r_{2k} \geq r_{2k-1} \geq \dots \geq r_0]$ be a nilpotent orbit of Type $B$, i.e. 
$|\{i\ |\ r_i = \alpha\}|$ is even for all {\bf even} integers $\alpha$.
Then ${}^L\mc{O}$ is a special orbit if and only if its dual partition defines a nilpotent orbit of Type $B$ or, equivalently,
$r_{2l+1} + r_{2l}$ is even for all $l \leq k-1$.
\item {\bf ${}^L\mf{g}$ of Type $C$:}\ \  Let ${}^L\mc{O} = [r_{2k+1} \geq r_{2k} \geq \dots \geq r_1]$ be a nilpotent orbit of Type $C$, i.e. $|\{i\ |\ r_i = \alpha\}|$ is even for all {\bf odd} integers $\alpha$.
Then ${}^L\mc{O}$ is a special orbit if and only if its dual partition defines a nilpotent orbit of Type $C$ or, equivalently,
$r_{2l+1} + r_{2l}$ is even for all $l \leq k$.
\item {\bf ${}^L\mf{g}$ of Type $D$:}\ \  Let ${}^L\mc{O} = [r_{2k+1} \geq r_{2k} \geq \dots \geq r_0]$ be a non-very even orbit of Type $D$, i.e. $|\{i\ |\ r_i = \alpha\}|$ is even for all {\bf even} integers $\alpha$. Then ${}^L\mc{O}$ is a special orbit if and only if its dual partition defines a nilpotent orbit of Type $C$ or, equivalently,
$r_{2l+1} + r_{2l}$ is even for all $l  \leq k$.
\end{itemize}
\end{proposition}

For any special orbit ${}^L\mc{O}$, the Lusztig's quotient $\overline{A}({}^L\mc{O})$ is given by:
\begin{proposition}[\cite{W2}, Proposition 2.2] \mbox{}\\
\begin{itemize}
\item {\bf ${}^L\mf{g}$ of Type $B$:}\ \ Let ${}^L\mc{O} = [r_{2k} \geq r_{2k-1} \geq \dots \geq r_0]$ be a special orbit. Separate all {\bf even} rows (which must be of the form $r_{2l+1} = r_{2l} = \alpha$), along with odd row pairs of the form $r_{2l} = r_{2l-1} = \beta$ and get
\begin{align*}
    {}^L\mc{O} = &[r_{2q}'' > r_{2q-1}'' \geq r_{2q-2}'' > \dots \geq r_2'' > r_1'' \geq r_0''] \cup
		\bigcup_i [\alpha_i, \alpha_i] \cup \bigcup_j [\beta_j, \beta_j],
\end{align*}
\item {\bf ${}^L\mf{g}$ of Type $C$:}\ \  Let ${}^L\mc{O} = [r_{2k+1} \geq r_{2k} \geq \dots \geq r_1]$ be a special orbit. Separate all {\bf odd} rows (which must be of the form $r_{2l+1} = r_{2l} = \alpha$), and even row pairs of the form $r_{2l} = r_{2l-1} = \beta$ and get
\begin{align*}
    {}^L\mc{O} = &[r_{2q+1}'' \geq r_{2q}'' > r_{2q-1}'' \geq \dots > r_3'' \geq r_2'' > r_1'']\cup
		\bigcup_i [\alpha_i, \alpha_i] \cup \bigcup_j [\beta_j, \beta_j],
\end{align*}
\item {\bf ${}^L\mf{g}$ of Type $D$:}\ \  Let ${}^L\mc{O} = [r_{2k+1} \geq r_{2k} \geq \dots \geq r_0]$ be a special, non-very even orbit. Separate all {\bf even} rows (which must be of the form $r_{2l+1} = r_{2l} = \alpha$), and all odd row pairs $r_{2l} = r_{2l-1} = \beta$ and get
  \begin{align*}
	{}^L\mc{O} = &[r_{2q+1}'' \geq r_{2q}'' > r_{2q-1}'' \geq \dots \geq r_2'' > r_1'' \geq r_0'']\cup
		\bigcup_i [\alpha_i, \alpha_i] \cup \bigcup_j [\beta_j, \beta_j],
	\end{align*}
\end{itemize}
Then $\overline{A}({}^L\mc{O}) = (\bb{Z}/2\bb{Z})^q$, regardless of the number of $\alpha_i$'s and $\beta_j$'s.
\end{proposition}

For any special orbit ${}^L\mathcal{O}$, consider
$\displaystyle {}^L\mc{O} = {}^L\mc{P} \cup {}^L\mc{Q},$
where ${}^L\mc{P}$ is the orbit of the same type as ${}^L\mathcal{O}$ without the $\alpha_i$'s, 
and ${}^L\mc{Q} := \bigcup_{i = 1}^{x} [\alpha_i, \alpha_i]$.
Note that ${}^L\mc{P}$ is even with $\overline{A}({}^L\mc{P}) 
\cong \overline{A}({}^L\mc{O})$ and the results in \cite{BV1985} 
hold for ${}^L\mc{P}$. More precisely, the coordinates of 
$\lambda_{\mc{O}}$ consists of:
\begin{itemize}
\item{\bf ${}^L\mf{g}$ of Type $B$:}\ \ integers coming from ${}^L\mc{P}$, half-integers coming from ${}^L\mc{Q}$.
\item{\bf ${}^L\mf{g}$ of Type $C$:}\ \ half-integers coming from ${}^L\mc{P}$, integers coming from ${}^L\mc{Q}$.
\item{\bf ${}^L\mf{g}$ of Type $D$:}\ \ integers coming from ${}^L\mc{P}$, half-integers coming from ${}^L\mc{Q}$.
\end{itemize}
Let ${}^L\mf{g}'$ be the Lie subalgebra of ${}^L\mf{g}$ whose roots are 
given by the roots $\alpha^{\vee}$ in ${}^L\mf{g}$ satisfying 
$\langle \alpha^{\vee}, \lambda_{\mc{O}} \rangle \in \bb{Z}$. 
We will study integral special unipotent representations
$\Pi({}^L\mc{O}')$, where ${}^L\mc{O}' \subset {}^L\mf{g}' = {}^L\mf{g}'_1 + {}^L\mf{g}'_2$
is an even orbit given by $\displaystyle {}^L\mc{O}' = {}^L\mc{P} + {}^L\mc{Q}$,
with
\begin{itemize}
\item{\bf ${}^L\mf{g}$ of Type $B$:}\ \ ${}^L\mc{P} \subset {}^L\mf{g}'_1$ is of Type $B$;\ \ \ ${}^L\mc{Q} \subset {}^L\mf{g}'_2$ is of Type $D$.
\item{\bf ${}^L\mf{g}$ of Type $C$:}\ \ ${}^L\mc{P} \subset {}^L\mf{g}'_1$ is of Type $C$;\ \ \ ${}^L\mc{Q} \subset {}^L\mf{g}'_2$ is of Type $C$.
\item{\bf ${}^L\mf{g}$ of Type $D$:}\ \ ${}^L\mc{P} \subset {}^L\mf{g}'_1$ is of Type $D$;\ \ \ ${}^L\mc{Q} \subset {}^L\mf{g}'_2$ is of Type $D$.
\end{itemize}
Note that $\overline{A}(\mc{O}) = \overline{A}({}^L\mc{O}) = \overline{A}({}^L\mc{P}) = \overline{A}({}^L\mc{O}')
= \overline{A}(\mc{O}')$.

We now study the unipotent representations $\Pi({}^L\mc{P})$ and $\Pi({}^L\mc{Q})$ individually.
For ${}^L\mc{Q}$, it has trivial Lusztig quotient, and {\it the} special unipotent representation attached to it is 
$\Pi({}^L\mc{Q}) =  \{\mathrm{Ind}_{GL(\alpha_1) \times \dots \times GL(\alpha_x)}^{G'_2}(\mathrm{triv} \boxtimes \dots \boxtimes \mathrm{triv})\}.$
Moreover, the (unique) left cell representation is $V^L(\mc{Q}) = J_{A_{\alpha_1-1} \times \dots \times A_{\alpha_x -1}}^{W(G'_2)}(\mathrm{sgn} \boxtimes \dots \boxtimes \mathrm{sgn})$.

On the other hand, suppose the left cell representation of $\mc{P}$ is $V^L(\mc{P}) = \displaystyle \bigoplus_{x \in \overline{A}(\mc{P})} \sigma'_x$. Then 
$$V^L(\mc{O}') = \bigoplus_{x \in \overline{A}(\mc{O}')} \sigma'_x \boxtimes J_{A_{\alpha_1-1} \times \dots \times A_{\alpha_x -1}}^{W(G'_2)}(\mathrm{sgn} \boxtimes \dots \boxtimes \mathrm{sgn}).$$

For $x \in \overline{A}(\mc{O}') \cong \overline{A}(\mc{O})$, let
\begin{align*}
\sigma_x &:= J_{W(G')}^{W(G)}(\sigma'_x \boxtimes J_{A_{\alpha_1-1} \times \dots \times A_{\alpha_x -1}}^{W(G'_2)}(\mathrm{sgn} \boxtimes \dots \boxtimes \mathrm{sgn}))\\
&= J_{W(G_1') \times A_{\alpha_1-1} \times \dots \times A_{\alpha_x -1}}^{W(G)}(\sigma'_x \boxtimes \mathrm{sgn} \boxtimes \dots \boxtimes \mathrm{sgn}).
\end{align*}
One can use Equation (4.6.5) of \cite{Lu1} to check that the right hand side is
non-zero and irreducible, and hence Step (VI) holds for all special $\mc{O}$. More explicitly, $\Pi({}^L\mc{O})$ is given by
$$\Pi({}^L\mc{O}) = \{ \mathrm{Ind}_{G_1' \times GL(\alpha_1) \times \dots \times GL(\alpha_x)}^{G}(J(\lambda',w_{\pi}\lambda') \boxtimes \mathrm{triv} \boxtimes \dots \boxtimes \mathrm{triv})\ |\ J(\lambda',w_{\pi}\lambda') \in \Pi({}^L\mc{P}) \}.$$

\subsection{Exceptional Lie algebras}
For exceptional special nilpotent orbits, we only focus on ${}^L\mc{O} \subset {}^L\mf{g}$ such that ${}^L\mc{O}$ is not even. Since there are no such orbits in $G_2$, we will only study exceptional Lie algebras of Type $E$ and $F$.


\medskip

In the following tables, we list all non-even special orbits ${}^L\mathcal{O}$ and their Lusztig-Spaltenstein dual $\mc{O}$.
In the third column, we get the subalgebra $\mf{g}'$ of $\mf{g}$ determined by $\lambda_{\mc{O}}$ (Step (I)). Then $\lambda_{\mc{O}}$ determines an even orbit ${}^L\mc{O}' \subset {}^L\mf{g}'$ (Step (II)), whose Lusztig-Spaltenstein dual $\mc{O}'$ is recorded in the fourth column. Afterwards, we can use \cite{Lu2} to compute $V^L(\mc{O}')$ (see Example \ref{eg:eg1} below), and then we have $V^L(\mc{O}) = \bigoplus_{\sigma' \in V^L(\mc{O}')} J_{W(G')}^{W(G)}(\sigma')$ (Step(VI)), which is given in the second last column of the table, with $\sigma_{sp} \in V^L(\mc{O})$ always appears first in the list. And the last column records the degree $\mathrm{deg}(\sigma_{sp}) = \mathrm{deg}(\sigma_{sp}')$. 

\medskip
The computations below are carried out by \texttt{LiE} \cite{LiE} and {\tt MATLAB}.

\vspace{5mm}

\subsubsection{Type $F_4$} The results for $F_4$ are as follows:\\

\begin{center}
\begin{tabular}{|c|c|c|c|c|c|c|c|} \hline
${}^L\mathcal{O}$
& $\mc{O}$
& $\mf{g}'$
& $\mc{O}'$
&  $V^L(\mc{O})$
& deg$(\sigma_{sp})$

 \cr
\hline
\hline
$\widetilde{A}_1$
& $F_4(a_1)$
& $B_4$
& $[711]$
& $4_2, 2_1$
& $1$
\cr
\hline 	
$A_1 + \widetilde{A}_1$
& $F_4(a_2)$
& $C_3+A_1$
& $[3^2]+[2]$
&  $9_1$
& $2$
\cr 	
\hline
$C_3$
& $A_2$
& $C_3+A_1$
& $[1^6]+[2]$
&  $8_2$
& $9$
\cr 	
\hline
\end{tabular}
\end{center}

\vspace{6mm}

\subsubsection{Type $E_6$} The results for $E_6$ are as follows:\\

\begin{center}
\begin{tabular}{|c|c|c|c|c|c|c|c|} \hline
${}^L\mathcal{O}$
& $\mc{O}$
& $\mf{g}'$
& $\mc{O}'$
&  $V^L(\mc{O})$
& deg$(\sigma_{sp})$

 \cr
\hline
\hline
$A_1$
& $E_6(a_1)$
& $A_5+A_1$
& $[6]+[1^2]$
&  $6_p$
& $1$
\cr
\hline 	
$2A_1$
& $D_5$
& $D_5$
& $[71^3]$
&  $20_p$
& $2$
\cr
\hline 	
$A_2+A_1$
& $D_5(a_1)$
& $A_5+A_1$
& $[41^2]+[1^2]$
&  $64_p$
& $4$
\cr
\hline 	

$A_2+2A_1$
& $A_4+A_1$
& $D_5$
& $[3^31]$
&  $60_p$
& $5$
\cr
\hline 	

$A_3$
& $A_4$
& $D_5$
& $[51^5]$
&  $81_p$
& $6$
\cr
\hline 	

$A_4+A_1$
& $A_2+2A_1$
& $A_5+A_1$
& $[21^4]+[1^2]$
&  $60_p'$
& $11$
\cr
\hline 	

$D_5(a_1)$
& $A_2+A_1$
& $D_5$
& $[2^21^6]$
&  $64_p'$
& $13$
\cr
\hline 	
\end{tabular}
\end{center}

\subsubsection{Type $E_7$} \label{sec-e7}
The results for $E_7$ are as follows:\\
\small{\begin{center}
\begin{tabular}{|c|c|c|c|c|c|c|c|} \hline
${}^L\mathcal{O}$
& $\mc{O}$
& $\mf{g}'$
& $\mc{O}'$
&  $V^L(\mc{O})$
& deg$(\sigma_{sp})$

 \cr
\hline
\hline
$A_1$
& $E_7(a_1)$
& $D_6+A_1$
& $[11,1]+[1^2]$
&  $7_a'$
& $1$
\cr
\hline 	
$2A_1$
& $E_7(a_2)$
& $D_6+A_1$
& $[91^3] + [2]$
&  $27_a$
& $2$
\cr
\hline 	
$A_2+A_1$
& $E_6(a_1)$
& $D_6+A_1$
& $[731^2]+[1^2]$
&  $120_a, 15_a'$
& $4$
\cr
\hline 	

$A_2+2A_1$
& $E_7(a_4)$
& $D_6+A_1$
& $[53^21] + [2]$
&  $189_b'$
& $5$
\cr
\hline 	

$A_3$
& $D_6(a_1)$
& $D_6+A_1$
& $[71^5] + [2]$
&  $210_a$
& $6$
\cr
\hline 	

$D_4(a_1)+A_1$
& $E_6(a_3)$
& $D_6+A_1$
& $[531^4] + [1^2]$
&  $405_a, 216_a'$
& $8$
\cr
\hline

$A_3+A_2$
& $D_5(a_1)+A_1$
& $D_6+A_1$
& $[3^31^3] + [2]$
&  $378_a'$
& $9$
\cr
\hline

$A_4+A_1$
& $A_4+A_1$
& $D_6+A_1$
& $[3^22^21^2]+[1^2]$
&  $512_a'$
& $11$
\cr
\hline 	

$D_5(a_1)$
& $A_4$
& $D_6+A_1$
& $[3^21^6] + [2]$
&  $420_a', 84_a'$
& $13$
\cr
\hline 	

$D_5 + A_1$
& $2A_2$
& $D_6+A_1$
& $[31^9] + [1^2]$
&  $168_a'$
& $21$
\cr
\hline 	

$D_6(a_1)$
& $A_3$
& $D_6+A_1$
& $[2^21^8] + [2]$
&  $210_a'$
& $21$
\cr
\hline 	

\end{tabular}
\end{center}}

\vspace{5mm}

\subsubsection{Type $E_8$} \label{sec-e8}
\normalsize{The results for $E_8$ are as follows:}\\
{\small
\begin{center}
\begin{tabular}{|c|c|c|c|c|c|c|c|} \hline
${}^L\mathcal{O}$
& $\mc{O}$
& $\mf{g}'$
& $\mc{O}'$
&  $V^L(\mc{O})$
& deg$(\sigma_{sp})$

 \cr
\hline
\hline
$A_1$
& $E_8(a_1)$
& $E_7+A_1$
& $E_7+[1^2]$
&  $8_z$
& $1$
\cr
\hline 	
$2A_1$
& $E_8(a_2)$
& $D_8$
& $[13,1^3]$
&  $35_x$
& $2$
\cr
\hline 	
$A_2+A_1$
& $E_8(a_4)$
& $E_7+A_1$
& $E_7(a_3)+[1^2]$
&  $210_x, 50_x$
& $4$
\cr
\hline 	

$A_2+2A_1$
& $E_8(b_4)$
& $D_8$
& $[93^21]$
&  $560_z$
& $5$
\cr
\hline 	

$A_3$
& $E_7(a_1)$
& $D_8$
& $[11,1^5]$
&  $567_x$
& $6$
\cr
\hline 	

$D_4(a_1)+A_1$
& $E_8(a_6)$
& $E_7+A_1$
& $E_7(a_5) + [1^2]$
&  $1400_x, 1050_x, 175_x$
& $8$
\cr
\hline

$A_3+A_2$
& $D_7(a_1)$
& $D_8$
& $[73^21^3]$
&  $3240_z$
& $9$
\cr
\hline

$A_4+A_1$
& $E_6(a_1)+A_1$
& $E_7+A_1$
& $D_5(a_1)+[1^2]$
&  $4096_z$
& $11$
\cr
\hline 	

$A_4 + 2A_1$
& $D_7(a_2)$
& $D_8$
& $[53^31^2]$
&  $4200_x, 840_x$
& $12$
\cr
\hline

$D_5(a_1)$
& $E_6(a_1)$
& $D_8$
& $[731^6]$
&  $2800_z, 700_{xx}$
& $13$
\cr
\hline 	

$D_5(a_1)+A_1$
& $E_7(a_4)$
& $E_7+A_1$
& $(A_3+A_2) + [2]$
&  $6075_x$
& $14$
\cr
\hline

$A_4+A_2+A_1$
& $A_6+A_1$
& $E_7+A_1$
& $(A_3+A_2+A_1) + [1^2]$
&  $2835_x$
& $14$
\cr
\hline

$D_6(a_1)$
& $E_6(a_3)$
& $D_8$
& $[531^8]$
&  $5600_z', 3200_x'$
& $21$
\cr
\hline 	

$A_6+A_1$
& $A_4+A_2+A_1$
& $E_7+A_1$
& $(A_2+3A_1) + [1^2]$
&  $2835_x'$
& $22$
\cr
\hline

$E_7(a_4)$
& $D_5(a_1)+A_1$
& $E_7+A_1$
& $(A_2+2A_1) + [2]$
&  $6075_x'$
& $22$
\cr
\hline

$D_7(a_2)$
& $A_4 + 2A_1$
& $D_8$
& $[3^22^21^6]$
&  $4200_x', 840_x'$
& $24$
\cr
\hline

$E_6(a_1)+A_1$
& $A_4+A_1$
& $E_7+A_1$
& $(A_2+A_1) + [1^2]$
&  $4096_z'$
& $26$
\cr
\hline

$E_7(a_3)$
& $A_4$
& $E_7+A_1$
& $A_2 + [2]$
&  $2268_x', 972_x'$
& $30$
\cr
\hline

$E_7(a_1)$
& $A_3$
& $E_7+A_1$
& $A_1 + [2]$
&  $567_x'$
& $46$
\cr
\hline

\end{tabular}
\end{center}}

\bigskip
Here is an example on how the above results are obtained:
\begin{example} \label{eg:eg1}
Let ${}^L\mc{O} = D_4(a_1) + A_1$ be a nilpotent orbit of Type $E_8$,
By calculating $\frac{1}{2}{}^LH$ explicitly, one can check that
${}^L\mf{g}'$ is of Type $E_7 + A_1$, and the coroots $\alpha^{\vee}$
such that $\langle \alpha^{\vee}, \frac{1}{2}{}^LH\rangle = 0$
forms a Lie subalgebra ${}^L\mf{l}'$ of Type $(A_5 + A_1) + 0$.
Then ${}^L\mc{O}' \subset {}^L\mf{g}'$ is the even orbit 
$${}^L\mc{O}' = \mathrm{Ind}_{{}^L\mf{l}'}^{{}^L\mf{g}'}(0) = D_4(a_1) + [2].$$
By looking at the tables of \cite{Car}, one can check that 
Lusztig-Spaltenstein dual of ${}^L\mathcal{O}$ and ${}^L\mc{O}'$
are $\mc{O} = E_8(a_6)$ and $\mc{O}' = E_7(a_5) + [1^2]$
respectively. Also, one can check from \cite{CM} that 
$\overline{A}(\mc{O}) = \overline{A}(\mc{O}') = S_3$.

We now study the left cell $V^L(\mc{O}')$: 
Firstly, note that $E_7(a_5) = \mathrm{Ind}_{E_6}^{E_7}(D_4(a_1))$.
By \cite{Lu2}, the special piece attached to $D_4(a_1)$
is equal to $\{D_4(a_1), A_3+A_1, 2A_2+A_1\}$, and their
Springer representation constitute the left cell
$$V^L(D_4(a_1)) = 80_s \oplus 60_s \oplus 10_s$$ 
(in fact,
$V^L(D_4(a_1))$ can also be obtained
directly from (4.11.2) of \cite{Lu1}, but this
perspective is useful in determining left cells
for classical orbits with large Lusztig quotient).

Let $[1^3]$, $[21]$, $[3]$ be the conjugacy classes
of $S_3$, then by Proposition 4.14 of \cite{BV1985} 
(which is valid since $\mc{O}'$ is even),
$$V^L(\mc{O}') =  J_{E_6 + A_1}^{E_7+A_1}(80_s \boxtimes \mathrm{sgn}) \oplus
J_{E_6+A_1}^{E_7+A_1}(60_s \boxtimes \mathrm{sgn}) \oplus J_{E_6+A_1}^{E_7+A_1}(10_s \boxtimes \mathrm{sgn}).$$
Using the notations of Step (III) in the first section, 
$\sigma'_{[1^3]} = J_{E_6 + A_1}^{E_7+A_1}(80_s \boxtimes \mathrm{sgn})$, $\sigma'_{[21]} = J_{E_6 + A_1}^{E_7+A_1}(60_s \boxtimes \mathrm{sgn})$ and $\sigma'_{[3]} = J_{E_6 + A_1}^{E_7+A_1}(10_s \boxtimes \mathrm{sgn})$.

By Proposition 3.18 of \cite{BV1983} and (4.13.3) of \cite{Lu1}, we have
\begin{align*}
V^L(\mc{O}) &=  J_{E_7+A_1}^{E_8}(\sigma'_{[1^3]}) \oplus
J_{E_7+A_1}^{E_8}(\sigma'_{[21]}) \oplus J_{E_7+A_1}^{E_8}(\sigma'_{[3]}) \\
&= J_{E_6 + A_1}^{E_8}(80_s \boxtimes \mathrm{sgn}) \oplus J_{E_6+A_1}^{E_8}(60_s \boxtimes \mathrm{sgn}) \oplus J_{E_6+A_1}^{E_8}(10_s \boxtimes \mathrm{sgn})\\
&= 1400_x \oplus 1050_x \oplus 175_x.
\end{align*}
Moreover, $\sigma_{sp} = 1400_x$ is the special representation corresponding to the $\mc{O}$ under the Springer correspondence.
Therefore, this verifies Step (VI) and Table A.2.4 for this orbit. 
%
\end{example}

We end the Appendix by mentioning the character formulas of the special unipotent 
representations $\Pi({}^L\mc{O})$ when $\mc{O}$ is equal to the three exceptional orbits, i.e. 
$A_4 + A_1$ in $E_7$, $A_4+A_1$ and $E_6(a_1) + A_1$
in $E_8$. 

From the Tables \ref{sec-e7} and \ref{sec-e8} above, one can check that for the $\mc{O}' \subset \mathfrak{g}'$
corresponding to these $\mc{O}$'s, $\overline{A}(\mc{O}) \cong \overline{A}(\mc{O}') \cong S_2$ and
$$V^L(\mc{O}') = \sigma_{[1^2]} \oplus \sigma_{[2]}.$$
However, their truncated inductions are given by
$$J_{W'}^{W}(\sigma_{[1^2]}') = \sigma_{sp},\ \ \ J_{W'}^{W}(\sigma_{[2]}') = 0,$$
so $V^L(\mc{O}) = \sigma_{sp}$ contains one irreducible representation only.
In other words, the second bullet point of Step (VI) does not hold for these orbits,
and one cannot express their character formulas in the form of Theorem III of \cite{BV1985}.

Nevertheless, the character formulas of $\Pi({}^L\mc{O})$
can still be obtained by using Step (IV), and the number of representations
is equal to the number of irreducible representations of 
$\overline{A}(\mc{O}) \cong \overline{A}(\mc{O}')$.



\end{document}